\documentclass[reqno]{amsart}
\usepackage{amsmath,amsthm,amssymb}
\usepackage{latexsym}
\usepackage{eucal}
\usepackage{fullpage}
\usepackage{amsmath,amsthm,amssymb}
\usepackage{latexsym}
\usepackage{eucal}
\usepackage{fullpage}
\usepackage{soul}

\def\cal{\mathcal}

\def\ge{\geqslant}\def\le{\leqslant}

\def\~{\widetilde}

\newtheorem{theorem}{Theorem}[section]
\newtheorem{lemma}{Lemma}[section]
\newtheorem{corollary}{Corollary}[section]

\def\cal{\mathcal}

\def\ge{\geqslant}
\def\le{\leqslant}
\def\~{\widetilde}

\begin{document}
\title[Orthogonal polynomials on the circle for the weight $w$ satisfying conditions $ w,w^{-1}\in {\rm
BMO}$
 ]{Orthogonal polynomials on the circle for the weight $w$ satisfying conditions $w,w^{-1}\in {\rm
BMO}$ }
\author{ S. Denisov, K. Rush }
\address{
\begin{flushleft}
Sergey Denisov: denissov@wisc.edu\\\vspace{0.1cm}
University of Wisconsin--Madison\\  Mathematics Department\\
480 Lincoln Dr., Madison, WI, 53706,
USA\vspace{0.1cm}\\and\\\vspace{0.1cm}
Keldysh Institute for Applied Mathematics, Russian Academy of Sciences\\
Miusskaya pl. 4, 125047 Moscow, RUSSIA\\
\vspace{1cm} Keith Rush: jkrush@wisc.edu
\\ \vspace{0.1cm}
University of Wisconsin--Madison\\  Mathematics Department\\
480 Lincoln Dr., Madison, WI, 53706, USA
\end{flushleft}
}\maketitle

\begin{abstract}
For the weight $w$ satisfying $w,w^{-1}\in {\rm BMO}(\mathbb{T})$,
we prove the asymptotics of $\{\Phi_n(e^{i\theta},w)\}$ in
$L^p[-\pi,\pi], 2\le p<p_0$ where $\{\Phi_n(z,w)\}$ are monic
polynomials orthogonal with respect to $w$ on the unit circle
$\mathbb{T}$.  Immediate applications include the estimates on the
uniform norm and asymptotics of the  polynomial entropies. The
estimates on higher order commutators between the Calderon-Zygmund
operators and BMO functions play the key role in the proofs of main
results.

{\it MSC AMS classification: 42C05, 33D45; key words: orthogonal
polynomials, weight, bounded mean oscillation.}

\end{abstract}\vspace{1cm}
\section{Introduction}

Let $\sigma$ be a probability measure on the unit circle. Define the
monic orthogonal polynomials $\{\Phi_n(z,\sigma)\}$ by requiring
\[
\deg \Phi_n=n, \,{\rm
coeff}(\Phi_n,n)=1,\,\int_{-\pi}^{\pi}\Phi_n(e^{i\theta},\sigma)\overline{\Phi_m(e^{i\theta},\sigma)}\,d\sigma\,=0\;,\quad
m<n,
\]
where ${\rm coeff}(Q,j)$ denotes the coefficient in front of $z^j$
in the polynomial $Q$. We can also define the orthonormal
polynomials  by the formula
\[
\phi_n(z,\sigma)=\frac{\Phi_n(z,\sigma)}{\|\Phi_n(e^{i\theta},\sigma)\|_{L^2_\sigma}}
\]
Later, we will need to use the following notation: for every
polynomial $Q_n(z)=q_nz^n+\ldots+q_0$ of degree at most $n$, we
introduce the $(\ast)$--operation:
\[
Q_n(z)\stackrel{(\ast)}{\longrightarrow} Q_n^*(z)=\bar{q}_0
z^n+\ldots+ \bar{q}_n
\]
This $(\ast)$ depends on $n$. In the paper, we use the shorthand
$\|f\|_p=\|f\|_{L^p(\mathbb{T})},
\|f\|_{L^p_w}=\left(\int_{\mathbb{T}}
|f(\theta)|^pwd\theta\right)^{1/p}$. $L^p$ stands for
$L^p(\mathbb{T})$ or $L^p[-\pi,\pi]$. The symbols $C, C_1$ are
reserved for absolute constants which value can change from one
formula to another.

The current paper is mainly motivated by two problems: the problem
of Steklov in the theory of orthogonal polynomials \cite{adt} and
the problem of the asymptotical behavior of the polynomial entropy
\cite{dk}.

The problem of Steklov \cite{1} consists in obtaining the sharp
estimates for $\|\phi_n(e^{i\theta},\sigma)\|_{L^\infty[-\pi,\pi]}$
assuming that the probability measure $\sigma$ satisfies $\sigma'\ge
\delta/(2\pi)$ a.e. on $[-\pi,\pi]$ and $\delta\in (0,1)$.  This
question attracted a lot of attention \cite{murman,Ger1,Ger2,Ger3,
Gol,rakh1,rakh2,2} and was recently resolved in \cite{adt}. In
particular, the following stronger result was proved
\begin{theorem}[\cite{adt}] Assume that the measure is given by the weight $w$: $d\sigma=wd\theta$. Let $p\in [1,\infty)$ and $C>C_0(p,\delta)$, then

\[
C_1(p,\delta)\sqrt n\le \sup_{w\ge \delta/(2\pi), \|w\|_1=1,
\|w\|_p\le C} \|\phi_n(e^{i\theta},w)\|_\infty\le C_2(p,\delta)\sqrt
n
\]
\end{theorem}
{\bf Remark.} If the measure $\sigma$ satisfies the Szeg\H{o}
condition \cite{szego}
\begin{equation}\label{szc}
\int_{-\pi}^\pi \log \sigma'(\theta)d\theta>-\infty
\end{equation}
then $\|\Phi_n\|_{L^2_\sigma}\sim 1$ and the polynomials $\phi_n$
and $\Phi_n$ are of the same size. In particular, $\phi_n$ can be
replaced by $\Phi_n$ in the previous Theorem.

{\bf Remark.}  In the formulation of the Steklov problem, the
normalization that $\sigma$ is a probability measure, i.e.,
\[
\int_{-\pi}^\pi d\sigma=1
\]
is not restrictive because of the following scalings: $
\phi_n(z,\sigma)=\alpha^{1/2} \phi_n(z,\alpha\sigma)$ and
$\Phi_n(z,\alpha\sigma)=\Phi_n(z,\sigma),\,\alpha>0$.

The previous Theorem handles all $p<\infty$ but not the case
$p=\infty$. That turns out to be essential: if the weight $w$ is
bounded, we get an improvement in the exponent.

\begin{theorem}(\cite{dn}, Denisov-Nazarov)
If $T\gg 1$, we have
\[
 \sup_{1\le w\le T}
\|\Phi_n(e^{i\theta},w)\|_{p_0}\le C(T), \,
p_0(T)=2+\frac{C_1}{T},\quad \sup_{1\le w\le T}
\|\Phi_n(e^{i\theta},w)\|_\infty\le C(T)n^{\frac 12-\frac{C}{T}}
\]
and, if $0<\epsilon\ll 1$,
\[
\sup_{1\le w\le 1+\epsilon} \|\Phi_n(e^{i\theta},w)\|_{p_0}\le
C(\epsilon), \, p_0(\epsilon)=\frac{C_2}{\epsilon}, \quad \sup_{1\le
w\le 1+\epsilon} \|\Phi_n(e^{i\theta},w)\|_\infty\le
C(\epsilon)n^{C\epsilon}
\]
\end{theorem}

The uniform bound on the $L^p$ norm suggests that maybe a stronger
result on the asymptotical behavior is true.  It is well-known that
for $\sigma$ in the Szeg\H{o} class (i.e., \eqref{szc} holds), the
following asymptotics is valid \cite{5}
\begin{equation}\label{asu}
\phi_n^*(e^{i\theta},\sigma)\stackrel{(*)}{\longrightarrow}
S(e^{i\theta},\sigma), \quad \int_{-\pi}^\pi
\left|\frac{\phi_n^*(e^{i\theta},\sigma)}{S(e^{i\theta},\sigma)}-1\right|^2d\theta\to
0, \quad n\to\infty
\end{equation}
where $\stackrel{(*)}{\longrightarrow}$ refers to weak-star
convergence and $S(z,\sigma)$ is the Szeg\H{o} function, i.e., the
outer function in $\mathbb{D}$ which satisfies
$|S(e^{i\theta},\sigma)|^{-2}=2\pi \sigma'(\theta), S(0,\sigma)>0$.
In particular, if $\sigma'\ge (2\pi)^{-1}\delta$, then
$\|\phi_n^*-S\|_2\to 0$. Recall that
$\phi_n(z,\sigma)=z^n\overline{\phi^*_n(z,\sigma)}, z\in
\mathbb{T}$.

The results stated above give rise to three questions: (a) What
upper estimate can we get assuming $w\in {\rm BMO}(\mathbb{T})$
\cite{stein} instead of $w\in L^\infty(\mathbb{T})$? Recall that
$L^\infty(\mathbb{T})\subset {\rm BMO}(\mathbb{T})\subset
\cap_{p<\infty}L^p(\mathbb{T})$. (b) Is it possible to relax the
Steklov condition $w\ge 1$? (c) Can one obtain an asymptotics of
$\{\phi_n^*\}$ in $L^p$ classes with $p>2$?\medskip

 The partial answers to these questions are
contained in the following
 Theorems which are the main results of the paper. We start with a
 comment about some notation. If $\alpha$ is a positive parameter, we write $\alpha\ll 1$ to
indicate the following: there is an absolute constant $\alpha_0$
(sufficiently small) such that $\alpha<\alpha_0$. Similarly, we
write $\alpha\gg 1$ as a substitute for: there is a constant
$\alpha_0$ (sufficiently large) so that $\alpha>\alpha_0$. The
symbol $\alpha_1\ll \alpha_2$ ($\alpha_1\gg \alpha_2$) will mean
$\alpha_1/\alpha_2\ll 1$ ($\alpha_1/ \alpha_2\gg 1$).

\begin{theorem}\label{th4}
If $w: \|w^{-1}\|_{BMO}\le s, \|w\|_{BMO}\le t$, then there is
$\Pi\in L^{p_0}[-\pi,\pi], p_0>2$ such that
\[
\lim_{n\to\infty} \|\Phi_n^*-\Pi\|_{p_0}=0
\]
and we have for $p_0$:
\begin{equation}\label{p}
p_0=\left\{
\begin{array}{cc}
2+\displaystyle \frac{C_1}{(st)\log^2 (st)}, & {\rm if }\,st\gg 1\\
\displaystyle \frac{C_2}{(st)^{1/4}}, & {\rm if}\, 0<st\ll 1
\end{array}
\right.
\end{equation}
We also have the bound for the uniform norm
\begin{equation}\label{nik}
\|\Phi^*_n\|_{\infty}\le C_{(st)}n^{1/p_0}
\end{equation}
where $C_{(u)}$ denotes a function of $u$.
\end{theorem}

In the case when an additional information is given, e.g., $w\in
L^\infty$ or $w^{-1}\in L^\infty$, this result can be improved.

\begin{theorem}\label{th1} Under the conditions of the previous
Theorem, we have
\begin{itemize}
\item{If $w\ge 1$, then $p_0$ can be taken as
\[
p_0=\left\{
\begin{array}{cc}
\displaystyle 2+\frac{C_1}{t\log t}, &\,{\rm if }\, t\gg 1\\
\displaystyle \frac{C_2}{\sqrt t}, &\,{\rm if}\, 0<t\ll 1
\end{array}
\right.
\]}

\item{If $w\le 1$, then we have
\[
p_0=\left\{
\begin{array}{cc}
\displaystyle 2+\frac{C_1}{s\log s}, &\,{\rm if }\, s\gg 1\\
\displaystyle \frac{C_2}{\sqrt s}, &\,{\rm if}\, 0<s\ll 1
\end{array}
\right.
\]}
We also have the bound for the uniform norm
\begin{equation}\label{nik1}
\|\Phi^*_n\|_{\infty}\le C_{(t,s)}n^{1/p_0}
\end{equation}
where $C_{(t,s)}$ depends on $t$ or $s$.
\end{itemize}

\end{theorem}

{\bf Remark.} It is clear that the allowed exponent $p_0$ is
decaying in $s$ and $t$ so it can be chosen larger than $2$ for all
values of $s$ and $t$.

{\bf Remark.} As we have already mentioned, the following scaling
invariance holds: $\Phi_n(z,\sigma)=\Phi_n(z,\alpha\sigma),$
$\alpha>0$. The $\rm BMO$ norm is $1$-homogeneous, e.g., $\|\alpha
w\|_{BMO}=\alpha\|w\|_{BMO}$, so the estimates in the
Theorem~\ref{th4} are invariant under scaling $w\to\alpha
w$.\smallskip

In the case when $w=C$, we get $\|w\|_{BMO}=\|w^{-1}\|_{BMO}=0$ and,
although $\Phi_n^*(z,w)=1$, we can not say anything about the size
of $\phi_n^*(z,w)$. The next Lemma explains how our results can be
generalized to $\{\phi_n^*\}$.
\begin{lemma}\label{pa}
In the Theorem \ref{th4}, if one makes an additional assumption that
$\|w\|_1=1$, then  $\|\phi_n^*-S\|_{p_0}\to 0$ with $p_0$ given by
\eqref{p}.
\end{lemma}
\begin{proof}
Indeed, Lemma \ref{balance} from Appendix shows that
\[
\int_{-\pi}^\pi \log wd\theta>-\infty
\]
and thus the sequence $\{\|\Phi_n\|_{2,w}\}$ has a finite positive
limit \cite{5,sim1}. Therefore, $\{\phi^*_n\} =
\left\{\frac{\Phi_n^*}{||\Phi_n||_{2, w}}\right\}$ has an $L^{p_0}$
limit by Theorem 1.3. By \eqref{asu}, $\{\phi_n^*\}$ converges
weakly to $S$ and therefore we have the statement of the Lemma with
$\Pi$ being a multiple of $S$.
\end{proof}

\medskip
The polynomial entropy is defined as
\[
E(n,\sigma)=\int_\mathbb T |\phi_n|^2\log|\phi_n|d\sigma
\]
where $\{\phi_n\}$ are orthonormal with respect to $\sigma$. In
recent years, a lot of efforts were made to understand the
asymptotics of $E(n,\sigma)$ \cite{ap1,ap2,ap3} as $n\to\infty$. In
\cite{dk}, the sharp lower and upper bounds were obtained for
$\sigma$ in the Szeg\H{o} class. In \cite{adt}, it was shown that
$E(n,w)$ can not exceed $C\log n$ if $w\ge 1$ and $w\in
L^p[-\pi,\pi],p<\infty$, and that this bound saturates. This leaves
us with very natural question: what are regularity assumptions on
$w$ that guarantee boundedness of $E(n,w)$? The following corollary
of Lemma \ref{pa} gives the partial answer.

\begin{corollary}\label{entr}If $w: w,w^{-1}\in {\rm BMO}(\mathbb{T})$ and $\|w\|_1=1$, then
\[
\lim_{n\to\infty} E(n,w)=-\frac{1}{4\pi}\int_{-\pi}^\pi \log(2\pi
w)d\theta
\]
\end{corollary}
So far, the only classes in which the  $E(n,w)$ was known to be
bounded were the Baxter's class \cite{sim1}: $d\sigma=wd\theta, \,
w\in W(\mathbb{T}), w>0$ ($W(\mathbb{T})$ denotes the Wiener
algebra) or the class given by positive weights with a certain
modulus of continuity \cite{szego}. Our conditions are obviously
much weaker and, in a sense, sharp.
\medskip

The structure of the paper is as follows: the main results are
proved in the next section, the Appendix contains auxiliary results
from harmonic analysis.\medskip

We use the following notation: $H$ refers to the Hilbert transform,
$P_{[i,j]}$ denotes the $L^2(d\theta)$ projection to the
$(i,\ldots,j)$ Fourier modes. Given two non-negative functions
$f_{1(2)}$ we write $f_1\lesssim f_2$ is there is an absolute
constant $C$ such that
\[
f_1\le Cf_2
\]
for all values of the arguments of $f_{1(2)}$. We define $\gtrsim$
similarly and say that $f_1\sim f_2$ if $f_1\lesssim f_2$ and
$f_2\lesssim f_1$ simultaneously. Given two operators $A$ and $B$,
we write $[A,B]=AB-BA$ for their commutator. If $w$ is a function,
then in the expression like $[w,A]$, the symbol $w$ is identified
with the operator of multiplication by $w$. The
Hunt-Muckenhoupt-Wheeden characteristic of the weight $w\in A_p$
will be denoted by $[w]_{A_p}$. For the basic facts about the BMO
class, $A_p$ and their relationship, we refer the reader to, e.g.,
the classical text \cite{stein}. If $A$ is a linear operator from
$L^p(\mathbb{T})$ space to $L^p(\mathbb{T})$, then $\|A\|_{p,p}$
denotes its operator norm.\vspace{1cm}

\section{Proofs of main results }

Before proving the main result, Theorem \ref{th4}, we need some
auxiliary Lemmas. We start with the following observation which goes
back to S. Bernstein \cite{bern,szego}.
\begin{lemma}For a monic polynomial $Q$ of degree $n$, we have:
\begin{equation}\label{cut}
Q(z)=\Phi_n(z,w)\quad {\it if\,\, and\,\, only\,\, if}\quad
P_{[0,n-1]}(wQ)=0.
\end{equation}
\end{lemma}
\begin{proof}It is sufficient to notice that \eqref{cut} is
equivalent to
\[
\int_{-\pi}^{\pi} Q(e^{i\theta})e^{-ij\theta}w(\theta)d\theta=0,
\quad j=0,\ldots,n-1
\]
which is the orthogonality condition.
\end{proof}

\begin{lemma}
If $f\in L^2(\mathbb{T})$ is real-valued function, $Q\in
L^\infty(\mathbb{T})$, then
\[ z^n\overline{P_{[0,n-1]}(f z^n\overline{Q})}=P_{[1,n]}(fQ)
\]
In particular, for a polynomial $P$ of degree at most $n$ with
$P(0)=1$, we have:
\[
P(z)=\Phi_n^*(z,w)\quad {\it if\,\, and\,\, only\,\, if}\quad
P_{[1,n]}(wP)=0.
\]
\end{lemma}
\begin{proof}
The first statement is immediate. The second one follows from the
Lemma above and the formula $\Phi_n=z^n\overline{\Phi_n^*},$ $z\in
\mathbb{T}$.
\end{proof}

We have the following three identities for $\Phi_n^*(z,w)$; the
first one was used in \cite{dn} recently. They are immediately
implied by the Lemma above.
\begin{equation}\label{odin}
\Phi_n^*=1+P_{[1,n]}\Bigl((1-\alpha w)\Phi_n^*\Bigr), \quad
\alpha\in \mathbb{R}
\end{equation}
\begin{equation}
\Phi_n^*=1+w^{-1}[w,P_{[1,n]}]\Phi^*_n
\end{equation}
\begin{equation}
\Phi_n^*=1-[w^{-1},P_{[1,n]}](w\Phi^*_n)
\end{equation}
Denote the higher order commutators recursively:
\[
{\bf C}_0=P_{[1,n]},\quad {\bf C}_1=[w,P_{[1,n]}],\quad {\bf
C}_l=[w,{\bf C}_{l-1}], \quad l=2,3,\ldots
\]
Define the multiple commutators of $w^{-1}$ and $P_{[1,n]}$ (in that
order!) by $\widetilde{{\bf C}}_j$.

\begin{lemma}\label{l1}
The following representations hold
\begin{equation}\label{r1}
 w^j P_{[1,n]} \Phi^*_n=\sum_{l=1}^{j}
\binom{j-1}{l-1} {\bf C}_l w^{j-l} \Phi^*_n
\end{equation}
and
\begin{equation}
w^{-j} P_{[1,n]} \Phi^*_n=-\sum_{l=0}^{j}  \binom{j}{l}
\widetilde{\bf C}_{l+1} w^{-(j-l)} (w\Phi^*_n)
\end{equation}
where $j=1,2,\ldots$.
\end{lemma}
\begin{proof}
We will prove \eqref{r1}, the other formula can be obtained in the
similar way. The case $j=1$ of this expression is our familiar
formula $wP_{[1,n]}\Phi^*_n = [w, P_{[1,n]}] \Phi^*_n$. Now the
proof proceeds by induction. Suppose we have
$$w^{k-1}P_{[1,n]}\Phi_n^* = \sum_{l=1}^{k-1} \binom{k-2}{l-1} {\bf C}_lw^{k-1-l}\Phi^*_n$$
Multiply both sides by $w$ and  write
\[
w^{k} P_{[1,n]} \Phi^*_n= \sum_{l=1}^{k-1} \binom{k-2}{l-1} w{\bf
C}_lw^{k-1-l}\Phi^*_n=\sum_{l=1}^{k-1} \binom{k-2}{l-1} \left({\bf
C}_{l+1}w^{k-1-l}\Phi^*_n+{\bf C}_{l}w^{k-l}\Phi^*_n\right)=
\]
\[
\sum_{l=1}^{k-1} \binom{k-2}{l-1}{\bf
C}_{l}w^{k-l}\Phi^*_n+\sum_{l=2}^{k} \binom{k-2}{l-2}{\bf
C}_{l}w^{k-l}\Phi^*_n=\sum_{l=1}^{k} \binom{k-1}{l-1} {\bf
C}_{l}w^{k-l}\Phi^*_n
\]
because $$\binom{k-1}{l-1} = \binom{k-2}{l-2}+\binom{k-2}{l-1}$$.
\end{proof}
Motivated by the previous Lemma, we introduce certain operators.
Given $f\in L^p$, define $\{y_j\}$ recursively by
\[
y_0=f,\, y_j=w^j+\sum_{l=0}^{j-1} \binom{j-1}{l} {\bf C}_{l+1}
y_{j-1-l}
\]
Then, we let
\[
z_j=w^{-j}-\sum_{l=0}^{j} \binom{j}{l}{\bf \widetilde
C}_{l+1}z_{j-l-1}
\]
where $z_{-1}=y_1, z_0=y_0=f$. Notice that for fixed $j$ both $y_j$
and $z_j$ are affine linear transformations in $f$. We can write
\[
y_j=y_j'+y_j''
\]
where
\[
y_0'=f,\quad  y_0''=0
\]
and, recursively,
\[
y_j'=\sum_{l=0}^{j-1} \binom{j-1}{l} {\bf C}_{l+1} y_{j-1-l}', \quad
y_j''=w^j+\sum_{l=0}^{j-1} \binom{j-1}{l} {\bf C}_{l+1} y_{j-1-l}''
\]
Similarly, we write $z_j=z_j'+z_j''$ where
\[
z_{-1}'=y_1',\quad z_{-1}''=y_1'', \quad z_{0}'=f, \quad z_{0}''=0
\]
and
\[
z_j'=-\sum_{l=0}^{j} \binom{j}{l}{\bf \widetilde
C}_{l+1}z_{j-l-1}',\quad  z_j''=w^{-j}-\sum_{l=0}^{j}
\binom{j}{l}{\bf \widetilde C}_{l+1}z_{j-l-1}'', \quad
\]
Let us introduce linear operators: $B_jf=y'_j,\, D_jf=z'_j$. We need
an important Lemma.\bigskip
\begin{lemma}\label{l78}
\[
w^j\Phi_n^*=y_j''+B_j\Phi_n^*, \quad
w^{-j}\Phi_n^*=z_j''+D_j\Phi_n^*
\]
\end{lemma}
\begin{proof}
This follows from
\[
w^j\Phi_n^*=w^j+w^jP_{[1,n]}\Phi_n^*, \quad
w^{-j}\Phi_n^*=w^{-j}+w^{-j}P_{[1,n]}\Phi_n^*
\]
and the previous Lemma.
\end{proof}
\smallskip
The next Lemma, in particular, provides the bounds for $B_j$ and
$D_j$.
\begin{lemma} \label{someest} Assume $w\ge 0, \|w\|_{BMO}=t, \|w^{-1}\|_{BMO}=s$, $\|w\|_1=1,$ and $p\in [2,3]$.
Then,
\[
\|B_j\|_{p,p}\le (Ctj)^j, \quad \|D_j\|_{p,p}\le (1+st)(Csj)^{j}
\]
Moreover,
\[
\|y_j''\|_p\le (C\widetilde tj)^j, \quad \|z_j''\|_p\le \widetilde
s\widetilde t(C\widetilde sj)^j
\]
with
\[
\widetilde t=\max\{t,1\},\quad \widetilde s=\max\{s,1\}
\]
\end{lemma}
\begin{proof}We will prove the estimates for $\|B_j\|_{p,p}$ and $\|y_j''\|_p$ only, the
bounds for $\|D_j\|_{p,p},\|z_j''\|_p$ are shown similarly. By
John-Nirenberg inequality (\cite{stein}, p.144), we get
\[
\int_{-\pi}^\pi |w-(2\pi)^{-1}|^{jp}d\theta\lesssim  j\int_0^\infty
x^{jp-1}\exp(-Cx/t)dx=j(Ct)^{jp}\Gamma(jp)\le (C_1tj)^{pj}
\]
where Stirling's formula was used for the gamma function
$\Gamma$.

Since
\[
|w|^{jp}\le (|w-(2\pi)^{-1}|+(2\pi)^{-1})^{jp}\le
C^{jp}(|w-(2\pi)^{-1}|^{jp}+1)
\]
we have
\[
\int_{-\pi}^{\pi} |w|^{jp}d\theta\le C^{jp}(1+(tj)^{jp})\le
(C_1\widetilde t j)^{jp}
\]
Lemma \ref{al1} yields
\[
\|y_j'\|_p\le \sum_{l=0}^{j-1} \frac{(j-1)!}{l!(j-1-l)!} (\widetilde
C(l+1)t)^{l+1}\|y'_{j-1-l}\|_p\le
(Ct)^{j}j!\sum_{k=0}^{j-1}\frac{(Ct)^{-k}}{k!}\|y_k'\|_p
\]
Divide both sides by $(Ct)^jj!$ and denote
$\beta_j=\displaystyle \frac{\|y'_j\|_p}{(Ct)^jj!}$. Then,
\[
\beta_j\le \sum_{l=0}^{j-1} \beta_l
\]
Since $\beta_0=\|f\|_p$, we have $\beta_j\le 3^j\|f\|_p$ by
induction and thus $\|y_j'\|_p\le (Ctj)^j\|f\|_p$. The estimates for
$\|y_j''\|_p, \|z_j'\|_p, \|z_j''\|_p$ can be obtained similarly.
\end{proof}\bigskip

\begin{lemma} \label{optim} If $\|w\|_1=1, \|w\|_{BMO}=t, \|w^{-1}\|_{BMO}=s$,
and $p\in [2,3]$, then
\[
\min_{l\in \mathbb{N}} \Bigl(\Lambda^{-l}\|B_l\|_{p,p}\Bigr)\le
\exp\left( -\frac{C\Lambda}{t}\right)
\]
and
\[
\min_{j\in \mathbb{N}} \Bigl(\epsilon^j \|D_j\|_{p,p}\Bigr)\leq
(1+st)\exp\left(-\frac{C}{\epsilon s}\right)
\]
provided that $\Lambda \gg t$ and $\epsilon \ll s^{-1}$.
\end{lemma}
\begin{proof}
By the previous Lemma, we have
\[
\Bigl(\Lambda^{-l}\|B_l\|_{p,p}\Bigr)\le\left(\frac{Ctl}{\Lambda}\right)^l
\]
Optimizing in $l$ we get $l^*\sim C\Lambda/(te)$ and it gives the
first estimate. The proof for the second one is identical.
\end{proof}
Now we are ready to prove the main results of the paper.
\begin{proof}{\it (Theorem \ref{th4}).}
Notice first that \eqref{nik} follows from the Nikolskii inequality
\[
\|Q\|_\infty<Cn^{1/{p_0}}\|Q\|_{p_0}, \quad \deg Q=n,\quad p_0\ge 2
\]
as long as the $L^{p_0}$ norms are estimated.
\bigskip

By scaling invariance, we can assume that $\|w\|_1=1$. We consider
two cases separately: $st\gg 1$ and $st\ll 1$. The proofs will be
different.\bigskip

{\bf 1. The case $st\gg 1$.}\smallskip

 Let
$p=2+\delta$ with $\delta<1$. Take two $n$-independent parameters
$\epsilon$ and $\Lambda$ such that $\epsilon s\ll 1$ and $\Lambda
t^{-1}\gg 1$. Consider the following sets \mbox{$\Omega_1=\{x:
w\le\epsilon\},$} $\Omega_2=\{x: \epsilon<w<\Lambda\}, \Omega_3=\{x:
w\ge\Lambda\}$. Notice that
\[
\epsilon s\ll 1, t\Lambda^{-1}\ll 1 \implies (\epsilon
s)(t\Lambda^{-1})\ll 1 \implies \epsilon\Lambda^{-1}\ll (st)^{-1}\ll
1\implies  \epsilon\ll \Lambda
\]
From \eqref{odin}, we have
\[
\Phi_n^*=1+P_{[1,n]}(1-w/\Lambda)\Phi^*_n
\]
The idea of our proof is to rewrite this identity in the form
\[
\Phi_n^*=f_n+\cal{O}(n)\Phi^*_n
\]
where $\|f_n\|_p<C(s,t)$ and $\cal{O}(n)$ is a contraction in $L^p$
for the suitable choice of $p$. To this end, we consider operators
\begin{eqnarray*}
\cal{O}_1(n)f=\epsilon^jP_{[1,n]}(1-w/\Lambda)\chi_{\Omega_1}\left(\frac{w}{\epsilon}\right)^j D_jf\\
\cal{O}_2(n)f=P_{[1,n]}(1-w/\Lambda)\chi_{\Omega_2}f\\
\cal{O}_3(n)f=\Lambda^{-l}P_{[1,n]}(1-w/\Lambda)(\Lambda/w)^l\chi_{\Omega_3}B_{l}f
\end{eqnarray*}
where $j$ and $l$ will be fixed later, they will be $n$-independent.
Let us estimate the $(L^p, L^p)$ norms of these operators.  Since
$\|P_{[1,n]}\|_{p,p}\le 1+C\delta$ (see Lemma \ref{hilb}),  we choose
$j$ and $l$ as in Lemma \ref{optim} to ensure
\[
\|\cal{O}_1(n)\|_{p,p}\le st\exp\left(-\frac{\widehat C}{\epsilon s}
\right)
\]
\[
\|\cal{O}_2(n)\|_{p,p}\le (1+C\delta)(1-\epsilon\Lambda^{-1})
\]
\[
\|\cal{O}_3(n)\|_{p,p}\le \exp\left(-\frac{\widehat C\Lambda}{t}
\right)
\]
Lemma \ref{l78} now yields
\[
\Phi_n^*=1+f_1(n)+f_3(n)+(\cal{O}_1(n)+\cal{O}_2(n)+\cal{O}_3(n))\Phi_n^*
\]
where
\[
f_1(n)=\epsilon^jP_{[1,n]}(1-w/\Lambda)\chi_{\Omega_1}\left(\frac{w}{\epsilon}\right)^j
z_j'', \quad
f_3(n)=\Lambda^{-l}P_{[1,n]}(1-w/\Lambda)(\Lambda/w)^l\chi_{\Omega_3}y_l''
\]
\[
{\rm Let }\,\,\quad f(n)= 1+f_1(n)+f_3(n)
\]
Then Lemma \ref{someest} provides the bound
\begin{equation}\label{ravn}
\|f(n)\|_p\le C(s,t)
\end{equation}
uniform in $n$. Denote
$\cal{O}(n)=\cal{O}_1(n)+\cal{O}_2(n)+\cal{O}_3(n)$ and select
parameters $\epsilon,\Lambda,\delta$ such that
$\|\cal{O}(n)\|_{p,p}<1-C\delta$. To do so, we first let
$\delta=c\epsilon\Lambda^{-1}$ with small positive absolute constant
$c$. Then, we consider
\[
st\exp\left(-\frac{\widehat C}{\epsilon s}
\right)+\exp\left(-\frac{\widehat C\Lambda}{t}
\right)=\frac{c_1\epsilon }{\Lambda}
\]
with $c_1$ again being a small constant. Now, solving equations
\[
st\exp(-\widehat C/(\epsilon s))=\exp(-\widehat C\Lambda/t), \quad
c_1\epsilon/\Lambda=2\exp(-\widehat C\Lambda/t)
\]
we get the statement of the Theorem. Indeed, we have two equations:
\[
\epsilon=\frac{\widehat Ct}{s(\widehat C\Lambda+t\log(s t))}
\]
and
\[
\frac{\Lambda}{t}=\frac{1}{\widehat
C}\left(C+\log(s\Lambda)+\log\Bigl(\frac{\Lambda}{t}+\frac{\log(st)}{\widehat
C }\Bigr)\right)
\]
Denote
\[
u=\widehat C\Lambda/t
\]
and then
\[
u=C+\log(st)+2\log u+\log\Bigl(1+\frac{\log (st)}{u}\Bigr)
\]
To find the required root, we restrict the range of $u$ to
$c_1\log(st) < u < c_2\log(st)$ for $c_1 \ll 1$, $c_2 \gg 1$.
Rewrite the equation above as
$$u-2\log u - \log\left(1+ \frac{\log(st)}{u}\right) = \log(st)+C$$
 Differentiating the left hand side in $u$, we see that ${\rm
 l.h.s.}'\sim 1$ within the given range. Therefore, there is exactly one solution $u$ and $u\sim \log(st)$.
 Then, since
$$\log\left(1+\frac{\log(st)}{u}\right)$$ is $O(1)$, we get
\[
u=\log(st)+2\log u+O(1)=\log(st)+2\log\log(st)+O(1)
\]
by iteration. Thus, \[ \frac{\epsilon}{\Lambda}=Ce^{-u}\sim
\frac{1}{st\log^2(st)}
\]
and $\delta\sim \frac{1}{st\log^2(st)}$. Now that we proved that
$\|\cal{O}(n)\|_{p,p}\le 1-C\delta<1$, we can rewrite
\[
\Phi_n^*=f(n)+\sum_{j=1}^\infty \cal{O}^j(n)f(n)
\]
and the series converges geometrically in $L^p$ with tail being
uniformly small in $n$ due to \eqref{ravn}.

To show that $\Phi_n^*$ converges in $L^p$ as $n\to\infty$, it is
sufficient to prove that  $\cal{O}^j(n)f(n)$ converges for each $j$.
This, however, is immediate. Indeed,
\[
P_{[1,n]}f\to P_{[1,\infty]}f,\quad {\rm as} \quad n\to\infty
\] in $L^q$ for all $f\in L^q, 1<q<\infty$. Since $w,w^{-1}\in
{\rm BMO}\subset \cap_{p\ge 1} L^p$ (\cite{stein}, this again
follows from John-Nirenberg estimate), we see that multiplication by
$w^{\pm j}$ maps $L^{p_1}$ to $L^{p_2}$ continuously by H\"older's
inequality provided that $p_2<p_1$ and $j\in \mathbb{Z}$. Therefore,
if $\mu_j\in L^\infty, j=1,\ldots,k$, then
\begin{equation}\label{sss}
\mu_1w^{\pm j_1}P_{[1,n]}\mu_2w^{\pm j_2}\ldots \mu_{k-1}w^{\pm
j_{k-1}}P_{[1,n]}\mu_kw^{\pm j_k}
\end{equation}
has the limit in each $L^p, p<\infty$ when $n\to\infty$. Notice that
each $f(n)$ and $\cal{O}^j(n)f(n)$ can be written as a linear
combination of expressions of type \eqref{sss} ($\{\mu_j\}$ taken as
the characteristic functions). Now that $\delta$ is chosen, we
define $p_0$ in the statement of the Theorem as $p_0=2+\delta$.
\bigskip

{2. \bf The case $st\ll 1$.}\smallskip

The proof in this case is much easier. Let us start with two
identities
\[
\Phi_n^*=1+w^{-1}[w,P_{[1,n]}]\Phi_n^*, \quad
\Phi_n^*=1+[P_{[1,n]},w^{-1}]w\Phi_n^*
\]
which can be recast as
\[
w\Phi_n^*=w+[w,P_{[1,n]}]\Phi_n^*, \quad
\Phi_n^*=1+[P_{[1,n]},w^{-1}]w\Phi_n^*
\]
Substitution of the first formula into the second one gives
\[
\Phi_n^*=1+[P_{[1,n]},w^{-1}]w+G_n\Phi_n^*
\]
where
\[
G_n=[P_{[1,n]},w^{-1}][w,P_{[1,n]}]
\]
We have
\[
\|1+[P_{[1,n]},w^{-1}]w\|_{p}\le C(s,t,p)
\]
and
\[
\|G_n\|_{p,p}\lesssim stp^4
\]
by Lemma \ref{ce}. Taking $p<p_0\sim (st)^{-1/4}$ we have that $G_n$
is a contraction. Now, the convergence of all terms in the geometric
series can be proved as before.
\end{proof}
Let us give a sketch of how the arguments can be modified to prove
Theorem \ref{th1}.
\begin{proof}{\it (Theorem \ref{th1}).}   Consider the case $w\ge 1$
first.\smallskip

{\bf 1. The case $t\gg 1$.}\smallskip

 The proof is identical
except that we can chose $\epsilon=1/2$ so that
$\Omega_1=\emptyset$. We get an equation for $\Lambda$
\[
\frac{ C}{\Lambda}=\exp\left(-\frac{\widehat C\Lambda}{t}\right),
\quad \Lambda=\widehat C^{-1}t(\log\Lambda-\log C)
\]
Denote $\widehat C\Lambda/t=u$, then
\[
u=\log t+\log u+O(1), \, u=\log t+\log\log t+O(1)
\]
and $\delta\sim (t\log t)^{-1}$.\bigskip

{\bf 2. The case $t\ll 1$.}\smallskip

We have
\[
\Phi_n^*=1+L_n\Phi_n^*, \quad L_nf=w^{-1}[w,P_{[1,n]}]f
\]
and Lemma \ref{ce} yields
\[
\|L_n\|_{p,p}\lesssim p^2t<0.5
\]
for $p<p_0=O(t^{-1/2})$.\bigskip

The case $w\le 1$ can be handled similarly. When $s$ is large, we
take $\Lambda=1$ in the proof of the previous Theorem and get an
equation for $\epsilon$:
\[
{C\epsilon}=s\exp\left(  -\frac{\widehat C}{\epsilon s} \right)
\]
Its solution for large $s$ gives the required asymptotics for
$\epsilon$ and, correspondingly, for $\delta$ and $p_0$. If $s$ is
small,  it is enough to consider the equation
\[
\Phi_n^*=1-[w^{-1},P_{[1,n]}]w\Phi_n^*
\]
where the operator $[w^{-1},P_{[1,n]}]w$ is contraction in $L^{p_0}$
for the specified $p_0$.
\end{proof}

Now we are ready to prove Corollary \ref{entr}.

\begin{proof}{\it (of Corollary \ref{entr}).}
The following inequality follows from the Mean Value Formula
\[
|x^2\log x-y^2\log y|\le C(1+x|\log x|+y|\log y|)|x-y|, \quad x,y\ge
0
\]
Since $w\in \cap_{p<\infty} L^p$, the Theorem \ref{th4} yields
\[
\int_{-\pi}^\pi
||\phi_n|^2\log|\phi_n|-|S|^2\log|S||wd\theta\lesssim
\int_{-\pi}^\pi(1+|\phi_n\log\phi_n|+|S\log
S|)||\phi_n^*|-|S||wd\theta\to 0, \, n\to\infty
\]
by applying  the trivial bound: $u|\log u|\le
C(\delta)(1+u^{1+\delta}), \delta>0$ and the generalized H\"older's
inequality to $|\phi_n|^{1+\delta} ({\rm or\,} |S|^{1+\delta}),
||\phi_n^*|-|S||$, and $w$. To conclude the proof, it is sufficient
to notice that
\[
\int_{-\pi}^{\pi} |S|^2\log |S|wd\theta=
-\frac{1}{4\pi}\int_{-\pi}^\pi \log(2\pi w)d\theta
\]
because $|S|^{-2}=2\pi w$.
\end{proof}
\vspace{1cm}

\section{Appendix}
In this Appendix, we collect some auxiliary results used in the main
text.

\begin{lemma}\label{hilb}
For every $p\in [2,\infty)$,
\begin{equation}
\|P_{[1,n]}\|_{p,p}\le 1+C(p-2)\,.
\end{equation}
\end{lemma}
\begin{proof}If $\cal{P}^+$ is the projection of $L^2(\mathbb{T})$ onto $H^2(\mathbb{T})$ (analytic Hardy space), then
\[
\cal{P}^+=0.5(1+iH)+P_{0}\,,
\]
where $H$ is the Hilbert transform on the circle and $P_{0}$ denotes
the Fourier projection to the constants, i.e.,
\begin{equation}
P_{0}f=(2\pi)^{-1}\int_{\mathbb{T}} f(x)dx
\end{equation}
We therefore have a representation
\begin{equation}\label{proj}
P_{[1,n]}=z\cal{P}^+z^{-1}-z^{n+1} \cal{P}^+
z^{-(n+1)}=0.5i(zHz^{-1}-z^{n+1} H z^{-(n+1)})+zP_{0}z^{-1}-z^{n+1}P_0z^{-(n+1)}\,.
\end{equation}
 Since
$\|H\|_{p,p}=\cot(\pi/(2p))$ \cite{pich}, we have
\begin{equation}
\|P_{[1,n]}\|_{p,p}\le \cot\left(\frac{\pi}{2p}\right)+2
\end{equation}
by triangle inequality. On the other hand, $\|P_{[1,n]}\|_{2,2}=1$ so by Riesz-Thorin theorem, we can interpolate between $p=2$ and, e.g., $p=3$ to get
\[
\|P_{[1,n]}\|_{p,p}\le 1+C(p-2), \quad p\in [2,3]\,.
\]
Noticing that $\cot(\pi/(2p))\sim p, \, p>3$, we get the statement of the Lemma.
\end{proof}
{\bf Remark.} In the proof above, we could have used the expression for the norm $\|\cal{P}^+\|_{p,p}$ obtained in \cite{iv}. 
\bigskip

 The proof of the following Lemma uses some standard results of Harmonic Analysis.\smallskip
\begin{lemma}\label{al1} If $\|w\|_{BMO}=t$ and $p\in [2,3]$, then we have
\[
\|{\bf C}_j\|_{p,p}\le   (Cjt)^j
\]
\end{lemma}
\begin{proof}

 Consider the following
operator-valued function
\[
F(z)=e^{zw}P_{[1,n]}e^{-zw}
\]
If we can prove that $F(z)$ is weakly analytic around the origin
(i.e., analyticity of the scalar function $\langle
F(z)f_1,f_2\rangle$ with fixed $f_{1(2)}\in C^\infty$), then
\[
F(z)=\frac{1}{2\pi i}\int_{|\xi|=\epsilon} \frac{F(\xi)}{\xi-z}d\xi,
\quad z\in B_\epsilon(0)
\]
understood in a weak sense. By induction, one can then easily show
the well-known formula
\[
{\bf C}_j=\partial^{j}F(0)=\frac{j!}{2\pi i}\int_{|\xi|=\epsilon}
\frac{F(\xi)}{\xi^{j+1}}d\xi
\]
which explains that we can control $\|{\bf C}_j\|_{p,p}$ by the size
of $\|F(\xi)\|_{p,p}$ on the circle of radius $\epsilon$. Indeed,
\[
\|{\bf C}_j\|_{p,p}=\sup_{f_{1(2)}\in C^\infty, \|f_1\|_p\le1,
\|f_2\|_{p'}\le1} |\langle {\bf C}_j f_1,f_2\rangle|\leq
\]
\[
\frac{j!}{2\pi}\sup_{f_{1(2)}\in C^\infty, \|f_1\|_p\le1,
\|f_2\|_{p'}\le1}\left|\int_{|\xi|=\epsilon} \frac{\langle
F(\xi)f_1,f_2\rangle}{\xi^{j+1}}d\xi \right|\le
\frac{j!}{\epsilon^j}\max_{|\xi|=\epsilon}\|F(\xi)\|_{p,p}
\]
The weak analyticity of $F(z)$  around the origin follows
 immediately from, e.g., the John-Nirenberg estimate (\cite{stein}, p.144). To bound $\|F\|_{p,p}$, we use the following well-known result (which
 is again an immediate corollary
from John-Nirenberg inequality, see, e.g., \cite{stein}, p.218).

{\it There is $\epsilon_0$ such that
\[
\|\widetilde w\|_{BMO}<\epsilon_0 \implies
[e^{\widetilde{w}}]_{A_p}\le [e^{\widetilde{w}}]_{A_2}<C,\quad  p>2
\]
}

The Hunt-Muckenhoupt-Wheeden Theorem (\cite{stein}, p.205), asserts
that
\begin{equation}\label{oh}
\sup_{[\widehat w]_{A_p}\le C} \|H\|_{(L^p_{\widehat
w}(\mathbb{T}),L^p_{\widehat w}(\mathbb{T}))}=
 \sup_{[\widehat w]_{A_p}\le C} \|\widehat w^{1/p}H\widehat
 w^{-1/p}\|_{p,p}=C(p)<\infty, \quad p\in [2,\infty)\,.
\end{equation}
We also have 
\[
\sup_{[\widehat w]_{A_p}\le C} \|P_{0}\|_{(L^p_{\widehat
w}(\mathbb{T}),L^p_{\widehat w}(\mathbb{T}))}=
 \sup_{[\widehat w]_{A_p}\le C, \|f\|_p\leq 1} \|\widehat w^{1/p}P_{0}\Bigl(\widehat
 w^{-1/p}f\Bigr)\|_{p}\leq (2\pi)^{-1} \sup_{[\widehat w]_{A_p}\le C, \|f\|_p\leq 1}\Bigl( \|f\|_p \|\widehat w\|_1^{1/p} \|\widehat w^{-1/p}\|_{p'}\Bigr)
\]
by H\"older's inequality. The last expression is bounded by a constant since
\[
[\widehat w]_{A_p}=\sup_Q \left(  \frac{1}{|Q|}\int_Q \widehat w dx \cdot \left(\frac{1}{|Q|} \int_Q \widehat w^{-p'/p}dx\right)^{p/p'}\right)\leq C\,,
\]
where $Q$ is any subarc of $\mathbb{T}$. Finally, taking  $\epsilon\ll t^{-1}$, we get the statement.
\end{proof}\bigskip

The following Lemma provides an estimate which is not optimal but it
is good enough for our purposes.\smallskip
\begin{lemma}\label{balance}
Suppose $w\ge 0,\|w\|_{BMO}=t, \|w^{-1}\|_{BMO}=s$, and $\|w\|_1=1$.
Then,
\[
(2\pi)^2\le \|w^{-1}\|_1\lesssim 1+(1+t)s
\]
\end{lemma}
\begin{proof}
Denote  $\|w^{-1}\|_1=M$. Then, by Cauchy-Schwarz inequality,
\[
2\pi\le \|w\|_1^{1/2}\|w^{-1}\|^{1/2}_1=M^{1/2}
\]
On the other hand, by John-Nirenberg estimate for $w^{-1}$,
\[
|\{\theta: |w^{-1}-(2\pi)^{-1}M|>\lambda\}|\lesssim
 \exp\left(-\frac{C\lambda}{s}\right)
\]
Choosing $\lambda=(4\pi)^{-1}M$, we get \begin{equation}\label{ll1}
 |\Omega^c|\lesssim  \exp\left(-\frac{CM}{s}\right)\lesssim
\left(\frac{s}{M}\right)^2,\quad  {\rm where}\quad
\,\Omega=\Bigl\{\theta:\frac{4\pi}{3M}\le w\le \frac{4\pi}{M}\Bigr\}
\end{equation}
Then, $\|w\|_1=1$ and therefore
\[
1=\int_{w\le (4\pi)/M}wd\theta+\int_{w>(4\pi)/M}wd\theta
\]
\begin{equation}\label{ll0}
\int_{w>(4\pi)/M}wd\theta\ge 1-8\pi^2 M^{-1}
\end{equation}
By John-Nirenberg inequality, we have \begin{equation}\label{ll2}
\|w-(2\pi)^{-1}\|_p<Ctp,\quad  p<\infty
\end{equation}
We choose $p=2$ in the last estimate and use Cauchy-Schwarz
inequality in \eqref{ll0} to get
\[1-8\pi^2 M^{-1}\le\int_{w>(4\pi)/M}wd\theta\le \|w\|_2 \cdot|\{\theta: w>4\pi/M\}|^{1/2} \le   \|w\|_2 \cdot|\Omega_c|^{1/2}  \lesssim \frac{(1+t)s}{M}
\]
where we used \eqref{ll1} and \eqref{ll2} for the last bound. So, $
M\lesssim (1+t)s+1. $
\end{proof}

\begin{lemma}\label{ce} For $p\in [2,\infty)$, we have
\[
 \|[w,P_{[1,n]}]\|_{p,p}\lesssim p^2\|w\|_{BMO}
\]
\end{lemma}
\begin{proof} The proof is standard but we give it here for
completeness. Assume $\|w\|_{BMO}=1$. By duality and formula
\eqref{proj}, it is sufficient to show that
\begin{equation}\label{upu1}
    \|[w,P_0]\|_{p,p}\le C(p-1)^{-1}, \quad p\in (1,2]
\end{equation}
and
\begin{equation}\label{upu}
    \|[w,H]\|_{p,p}\le C(p-1)^{-2}, \quad p\in (1,2]\,.
\end{equation}
For \eqref{upu1}, we write
\[
\|w\int_{\mathbb{T}} fdx-\int_{\mathbb{T}} wfdx\|_p\le \|f\|_1\|w-\langle w\rangle_{\mathbb{T}}\|_p+\|f\|_p\|w-\langle w\rangle_{\mathbb{T}}\|_{p'}
\]
where
\[
\langle w\rangle_{\mathbb{T}}=\frac{1}{2\pi}\int_{\mathbb{T}} wdx\,.
\]
From  John-Nirenberg theorem, we have 
\[
\|w-\langle w\rangle_{\mathbb{T}}\|_{p'}\lesssim p'\|w\|_{BMO}, \quad p'>2\,,
\]
which proves \eqref{upu1}.
To prove \eqref{upu}, we will interpolate between two bounds: the standard
Coifman-Rochberg-Weiss theorem for $p=2$ (\cite{crw},\cite{stein})
\begin{equation}\label{raz}
\|[H,w]\|_{2,2}\le C
\end{equation}
and the following estimate
\begin{equation}\label{dva}
|\{x:|([H,w]f)(x)|>\alpha\}|\le C \int_{\mathbb{T}}
\frac{|f(t)|}{\alpha}\left(1+\log^+\left(\frac{|f(t)|}{\alpha}\right)\right)dt
\end{equation}
(See \cite{perez}, the estimate was obtained on $\mathbb{R}$  for
smooth $f$ with compact support. The proof, however, is valid for
$\mathbb{T}$ as well and, e.g., piece-wise smooth continuous $f$).
Assume a smooth $f$ is given and denote $\lambda_f(t)=|\{x:
|f(x)|>t\}|, t\ge 0$. Take $A>0$ and consider $f_A=f\cdot
\chi_{|f|\le A}+A\cdot {\rm sgn}f\cdot \chi_{|f|>A}$, $g_A=f-f_A$.
Let $T=[H,w]$. Then,
\[
\|Tf\|_p^p=p\int_0^\infty t^{p-1}\lambda_{Tf}(t)dt\le p\int_0^\infty
t^{p-1}\lambda_{Tf_A}(t/2)dt+ p\int_0^\infty
t^{p-1}\lambda_{Tg_A}(t/2)dt=I_1+I_2
\]
Let $A=t$. From  Chebyshev inequality and \eqref{raz}, we get
\[
I_1\lesssim \int_0^\infty t^{p-3}\|f_A\|_2^2dt=2\int_0^\infty
t^{p-3}\int_0^A \xi \lambda_f(\xi)d\xi dt\lesssim
(2-p)^{-1}\int_0^\infty \xi^{p-1}\lambda_f(\xi)d\xi\lesssim
(2-p)^{-1}\|f\|_p^p
\]
For $I_2$, we use \eqref{dva} (notice that $g_A$ is continuous and
piece-wise smooth)
\[
I_2\lesssim -\int_0^\infty t^{p-1}\int_0^\infty
\frac{\xi}{t}\left(1+\log^+\frac{\xi}{t}\right)d\lambda_{g_A}(\xi)\lesssim
\]
\[
\|f\|_p^p+\int_0^\infty t^{p-1}\int_{2t}^\infty
t^{-1}\Bigl(1+\log^+((\tau-t)/t)\Bigr)\lambda_{f}(\tau)d\tau\lesssim
\|f\|_p^p \int_0^1
\xi^{p-2}\left(1+\log^+\frac{1-\xi}{\xi}\right)d\xi
\]
We have
\[
\int_0^{1/2}
\xi^{p-2}\left(1+\log^+\frac{1-\xi}{\xi}\right)d\xi\lesssim
\int_2^\infty u^{-p}\log u du\lesssim \int_0^\infty e^{-\delta
t}tdt\lesssim \delta^{-2}
\]
with $\delta=p-1$.

\end{proof}

{\Large \part*{Acknowledgement}} The work of SD done in the second
part of the paper was supported by RSF-14-21-00025 and his research
on the rest of the paper was supported by the grant NSF-DMS-1464479.
The research of KR was supported by the RTG grant NSF-DMS-1147523.

\end{document}